\documentclass{amsart}

\newtheorem{theorem}{Theorem}[section]
\newtheorem{lemma}[theorem]{Lemma}
\newtheorem{proposition}[theorem]{Proposition}
\newtheorem{corollary}[theorem]{Corollary}
\newtheorem{question}[theorem]{Question}

\theoremstyle{definition}

\newtheorem{remark}[theorem]{Remark}

\newenvironment{proofof}[1]{\noindent{\it Proof of
#1.}}{\hfill$\square$\\\mbox{}}

\def\mz{{\mathbb{Z}}}
\def\mn{{\mathbb{N}}}
\def\field{{\mathbb{F}}}
\def\support{{\mathrm{supp}}}
\def\sepbeta{\beta_{\mathrm{sep}}}
\def\rank{\mathrm{rank}}

\usepackage{amscd,amssymb}

\begin{document}

\title[Degree bound for separating invariants of abelian groups]
{Degree bound for separating invariants of abelian groups}

\author[M. Domokos]
{M. Domokos}
\address{MTA Alfr\'ed R\'enyi Institute of Mathematics,
Re\'altanoda utca 13-15, 1053 Budapest, Hungary}
\email{domokos.matyas@renyi.mta.hu}

\thanks{This research was partially supported by OTKA K101515.}



\begin{abstract} It is proved that the universal degree bound for separating polynomial invariants of a finite abelian group (in non-modular characteristic) is strictly smaller than the universal degree bound for generators of polynomial invariants, unless the goup is cyclic or is the direct product of $r$ even order cyclic groups where the number of two-element direct factors is not less than the integer part of the half of $r$. A characterization of separating sets of monomials is given in terms of zero-sum sequences over  abelian groups. 
\end{abstract}

\maketitle


\section{Introduction}\label{sec:intro}

Let $G$ be a finite group and $\field$ an algebraically closed field. A {\it $G$-module} is a finite dimensional $\field$-vector space $V$ endowed with an action of $G$ on $V$ 
via linear transformations. In other words, the $G$-module consists of the pair $(V,\rho)$ where $\rho$ is a group homomorphism $G\to GL(V)$. 
The coordinate ring $\mathcal{O}(V)$ of $V$ contains the subalgebra 
\[\mathcal{O}(V)^G:=\{f\in \mathcal{O}(V)\colon f(gv)=f(v) \quad \forall v\in V,\forall g\in G\}\] 
of $G$-invariants. 
For $f\in \mathcal{O}(V)$ and $g\in G$ write $g\cdot f\in \mathcal{O}(V)$ for the function $v\mapsto f(g^{-1}v)$. This way we get an action of $G$ on $\mathcal{O}(V)$ via 
$\field$-algebra automorphisms, and $\mathcal{O}(V)^G=\{f\in \mathcal{O}(V)\colon g\cdot f=f\quad \forall g\in G\}$. 
Choosing a basis $x_1,\dots,x_k$ in the dual space $V^*$ of $V$, the coordinate ring $\mathcal{O}(V)$ is identified with the polynomial algebra $\field[x_1,\dots,x_k]$, 
on which $G$ acts via linear substitutions of the variables.  

Following Definition 2.3.8 in \cite{derksen-kemper}, we call subset $S\subset \mathcal{O}(V)^G$  a {\it separating set of invariants} if whenever for $v,w\in V$ we have 
$f(v)= f(w)$ for all $f\in S$, then 
$h(v)=h(w)$ for all $h\in \mathcal{O}(V)^G$. Clearly if $v$ and $w$ belong to the same $G$-orbit in $V$, then $h(v)=h(w)$ holds for all $h\in\mathcal{O}(V)^G$. 
It is well known that the finiteness of  $G$ implies the converse as well: if $v$ and $w$ have different $G$-orbits, then there exists an $h\in \mathcal{O}(V)^G$ with $h(v)\neq h(w)$. So $S\subset O(V)^G$ is a separating set if and only if for any $v,w\in V$ with $Gv\neq Gw$ there is an $f\in S$ such that $f(v)\neq f(w)$. 
For a survey on separating sets of invariants see \cite{kemper}.

Since the $G$-action preserves the standard grading on $\field[x_1,\dots,x_k]$, the algebra $\mathcal{O}(V)^G$ is generated by homogeneous elements. 
Write $\beta(G,V)$ (respectively $\sepbeta(G,V)$) for the minimal positive integer $k$ such that $\mathcal{O}(V)^G$ contains a generating set (respectively separating set)  consisting of homogeneous elements of degree at most $k$. 
Moreover, set 
\[\beta(G):=\sup_V\{\beta(G,V)\}\quad \mbox{ and } \quad \sepbeta(G):=\sup_V\{\sepbeta(G,V)\}\] 
where the supremum above is taken over all $G$-modules $V$. 
The number $\sepbeta(G)$ was introduced and studied in \cite{kohls-kraft}, inspired by the number $\beta(G)$ first appearing in \cite{schmid}. 
Obviously $\sepbeta(G,V)\le \beta(G,V)$ and hence $\sepbeta(G)\le \beta(G)$. 
When the characteristic of $\field$ does not divide the group order $|G|$, we have $\beta(G)\le |G|$ (see \cite{noether} for $\mathrm{char}(\field)=0$ 
and \cite{fleischmann},\cite{fogarty} for positive non-modular characteristic). 
One nice feature of $\sepbeta(G)$ is that the inequality $\sepbeta(G)\le |G|$ holds also in the modular case $\mathrm{char}(\field)\mid |G|$ as well, see 
Corollary 3.9.14 in \cite{derksen-kemper}. In comparison we mention that when $\mathrm{char}(\field)$ divides $|G|$ we have $\beta(G)=\infty$ by \cite{richman}.  
However, as far as we know, not much is said in the literature about the following question: 
\begin{question}\label{question:strict} 
Is $\sepbeta(G)$ typicallly strictly smaller than $\beta(G)$ in the non-modular case $\mathrm{char}(\field)\nmid |G|$? 
\end{question} 
A difficulty in answering Question~\ref{question:strict} is that the exact value of the Noether number is known only for a very limited class of groups, see for example 
\cite{CD2}, \cite{CD3}, \cite{cziszter-domokos-geroldinger}. 
It is shown in \cite{cziszter} that for the non-abelian semidirect product $C_p\rtimes C_3$ (where $p$ is a prime) and $\mathrm{char}(\field)=0$ we have 
$\beta(C_p\rtimes C_3)=p+2$ whereas $\sepbeta(C_p\rtimes C_3)=p+1$. 

In the present paper we shall deal with abelian groups. 
Our main result Theorem~\ref{thm:main} implies  that for abelian groups the answer to Question~\ref{question:strict} is yes. More precisely, Corollary~\ref{cor:strictinequality} 
asserts that when $G$ is abelian, $\sepbeta(G)=\beta(G)$ implies that $G$ is cyclic or $G$ is the direct product of $r$ cyclic groups of even order, where at least $\lfloor \frac r2\rfloor$ of the cyclic factors has order $2$. 

A  interesting special feature of the case of abelian groups is that  the investigation of separating invariants can be tied up with the theory of zero-sum sequences over abelian groups. 
Given a finite abelian group $G$ (written additively) and an ordered sequence $a_1,\dots,a_k$ of elements of $G$ (repetition is allowed) set 
\[\mathcal{G}(a_1,\dots,a_k):=\{(m_1,\dots,m_k)\in \mz^k\colon \sum m_ia_i=0\in G\}.\] 
This is a subgroup of the free abelian group $\mz^k$. It contains the submonoid 
\[\mathcal{B}(a_1,\dots,a_k):=\mn_0^k\cap \mathcal{G}(a_1,\dots,a_k).\] 
Denote by $e_i$ the $i$th standard basis vector in $\mz^k$.  Clearly $\mathrm{ord}_G(a_i)e_i\in \mathcal{B}(a_1,\dots,a_k)$, where $\mathrm{ord}_G(a_i)$ is the order of $a_i$ in $G$.  Since for any 
$m\in \mathcal{G}(a_1,\dots,a_k)$ there exist non-negative integers $t_1,\dots,t_k\in\mn_0$ with 
$m+\sum t_i \mathrm{ord}_G(a_i)e_i\in \mathcal{B}(a_1,\dots,a_k)$, it follows that $\mathcal{G}(a_1,\dots,a_k)$ is the quotient group of the monoid 
$\mathcal{B}(a_1,\dots,a_k)$. In particular, the abelian group $\mathcal{G}(a_1,\dots,a_k)$ is generated by its submonoid  $\mathcal{B}(a_1,\dots,a_k)$. 
In the special case when $a_1,\dots,a_k$ are distinct and $\{a_1,\dots,a_k\}=G$, we recover the monoid $\mathcal{B}(G)$ of zero-sum sequences over $G$, a well studied object in arithmetic combinatorics.  In particular, the {\it Davenport constant}  $\mathsf D(G)$ is defined as the maximal length of an atom in the monoid $\mathcal{B}(G)$, 
where for $s\in\mathcal{B}(G)\subset \mn_0^{|G|}$ the {\it length} of $s$ is $|s|=\sum_{g\in G}s_g$.  More generally, 
the study of the monoid $\mathcal{B}(G_0)$ of zero-sum sequences over an arbitrary subset $G_0$ of $G$ has an extensive literature, see Proposition 2.5.6 in \cite{Ge-HK06a} 
for the first abstract algebraic properties of the monoid $\mathcal{B}(G_0)$, or \cite{Pl-Tr15a} for recent combinatorial work on $\mathsf D(G_0)$ (for some very special subset $G_0$).

From now on we assume that $G$ is a finite abelian group and the characteristic of the base field $\field$ does not divide $|G|$. 
Then $V$ decomposes as a direct sum 
\[V=V_1\oplus\dots\oplus V_k\]  
of $1$-dimensional $G$-modules. Accordingly the variables in $\mathcal{O}(V)=\field[x_1,\dots,x_k]$ will be chosen to be $G$-eigenvectors, so there exist characters $\chi_i\in \widehat G=\hom(G,\field^\times)$ such that $g\cdot x_i=\chi_i(g)x_i$ for $i=1,\dots,k$. For 
$m\in \mn_0^k$ write $x^m=x_1^{m_1}\cdots x_k^{m_k}$. Each monomial spans a $G$-invariant subspace in $\mathcal{O}(V)$, and 
$g\cdot x^m=(\prod_{i=}^k \chi_i(g)^{m_i})x^m$. It follows that $\mathcal{O}(V)^G$ is spanned by $G$-invariant monomials, namely 
\begin{equation}\label{eq:monomial-decomp}\mathcal{O}(V)^G=\bigoplus_{m\in \mathcal{B}(\chi_1,\dots,\chi_k) } \field x^m. 
\end{equation}
Note that here  we use the notation introduced in the above paragraph for the finite abelian group
 $\widehat G$ which is a isomorphic to $G$. 
 A consequence of \eqref{eq:monomial-decomp} is the equality 
 \begin{equation}\label{eq:noether=davenport}
\beta(G)=\mathsf D(G)
\end{equation}
which was used in \cite{schmid}, and later  in \cite{finklea_etal} or in 
\cite{cziszter-domokos-geroldinger}.  
In view of the above connection between the Noether number $\beta(G)$ and the Davenport constant $\mathsf D(G)$ it is natural to ask for the meaning of $\sepbeta(G)$ in terms of zero-sum sequences. This is the second motivation of the present paper. In Theorem~\ref{thm:monomial-separating} we provide a characterization of 
separating sets of monomials and zero-sum sequences over $G$, yielding a characterization of $\sepbeta(G)$ purely in terms of zero-sum sequences over $G$ (see Corollary~\ref{cor:betasep}). 
This is done in Section~\ref{sec:characterization}, and Corollary~\ref{cor:betasep} is used in Section~\ref{sec:degree bounds} to derive our main result 
Theorem~\ref{thm:main}. 

 We finish the introduction by mentioning some prior works related to separating invariants of finite diagonal groups. Namely, a separating set of monomials in $\mathcal{O}(V)^G$ is constructed in Proposition 5 of \cite{neusel-sezer}.  An algorithm to  produce invariant monomials that generate the field of rational invariants is described in \cite{hubert-laban}. 
The focus of present paper is on degree bounds for separating invariants, and therefore it is sufficient to deal with invariant monomials. A different current research direction is the study of the minimal cardinality of a separating system, see for example 
\cite{dufresne-jeffries}. 
 

\section{Characterization of separating sets of monomials}\label{sec:characterization}

Let $G$ be a finite abelian group, and let $V=V_1\oplus  \dots \oplus  V_k$ be a $k$-dimensional $G$-module as in Section~\ref{sec:intro}, so $\mathcal{O}(V)=\field[x_1,\dots,x_k]$ with $g\cdot x_i=\chi_i(g)x_i$ for $i=1,\dots,k$.  
For $m\in \mn_0^k$ set $\support(m):=\{i\in \{1,\dots,k\}\colon m_i\neq 0\}\subset \{1,\dots,k\}$. Similarly, when $x^m$ is a monomial, we shall use the notation $\support(x^m)$ for $\support(m)$.  Given a subset $J\subset \{1,\dots,k\}$ and a set $M\subset \mn_0^k$ we write $M_J:=\{m\in M\colon \support(m)\subset J\}$. 
 
The {\it Helly  dimension} $\kappa(G)$ of $G$ was defined in \cite{domokos} as the minimal positive integer $k$ such that any set of cosets in $G$ with empty intersection contains a subset of at most $k$ cosets with empty intersection. It was shown in \cite{domokos-szabo} that $\kappa(G)$ is one bigger than the minimal number of generators of the finite abelian group $G$ (the {\it rank} of $G$).

\begin{theorem}\label{thm:monomial-separating} 
For a subset $M\subset \mathcal{B}(\chi_1,\dots,\chi_k)$ the following are equivalent: 
\begin{itemize}
\item[(i)] $\{x^m\colon m\in M\}$ is a separating subset in $\mathcal{O}(V)^G$. 
\item[(ii)] For all subsets $\{j_1<\dots <j_s\}=J \subset \{1,\dots,k\}$, the abelian group $\mathcal{G}(\chi_{j_1},\dots,\chi_{j_s})$ is generated by 
$M_J$. 
\item[(iii)] For all subsets $\{j_1<\dots <j_s\}=J \subset \{1,\dots,k\}$ with $|J|\le \kappa(G)$, the abelian group $\mathcal{G}(\chi_{j_1},\dots,\chi_{j_s})$ is generated by 
$M_J$. 
\end{itemize}
\end{theorem}
 
The proof will be split into a couple of statements. 
Consider the $G$-module direct summand  $V_J:=\bigoplus_{j\in J} V_j$ of $V$, where $J\subset \{1,\dots,k\}$. Its coordinate ring $\mathcal{O}(V_J)$ is an algebra retract of $\mathcal{O}(V)$: it is the subalgebra generated by the variables $\{x_j\colon j\in J\}$.  For $v\in V$ we write $v_J$ for the component of $V$ in the direct summand 
$V_J$ of $V$. The statement and proof of Lemma~\ref{lemma:helly} below remain valid when the finite group  $G$ is not assumed to be abelian and the direct summands 
$V_j$ are not assumed to be $1$-dimensional. 

\begin{lemma}\label{lemma:helly} 
Assume that $k\ge \kappa(G)$ and for all $J\subset \{1,\dots,k\}$ with $|J|=\kappa(G)$ we are given a separating subset $S_J$ in $\mathcal{O}(V_J)^G$. 
Then their union $S:=\bigcup_J S_J$ is a separating subset in $\mathcal{O}(V)^G$. 
\end{lemma}

\begin{proof} Suppose that for $v,w\in V$ we have $f(v)=f(w)$ for all $f\in S$. Then in particular, for a fixed $J\subset \{1,\dots,k\}$ with $|J|=\kappa(G)$ we have 
$f(v_J)=f(w_J)$ for all $f\in S_J$. Since $S_J$ is a separating subset in $\mathcal{O}(V_J)^G$, we conclude that $Gv_J=Gw_J$.  
This holds for all $J\subset \{1,\dots,k\}$ with $|J|=\kappa(G)$, hence by Lemma 4.1 in \cite{domokos-szabo} we get that $Gv=Gw$. 
\end{proof} 

\begin{proposition}\label{prop:(ii)implies(i)} 
Let $M$ be a subset of  $\mathcal{B}(\chi_1,\dots,\chi_k)$ such that for all $J=\{j_1<\dots<j_s\}\subset \{1,\dots,k\}$ the abelian group 
$\mathcal{G}(\chi_{j_1},\cdots,\chi_{j_s})$ is generated by $M_J$. Then $\{x^m\colon m\in M\}$ is a separating set in $\mathcal{O}(V)^G$. 
\end{proposition}

\begin{proof} Take $v,w\in V$ such that $x^m(v)=x^m(w)$ for all $m\in M$. The assumption says in particular that $M_{\{j\}}$ generates $\mathcal{G}(\chi_j)$ for $j=1,\dots,d$. 
Since $\mathcal{G}(\chi_j)$ is the subgroup of $\mz$ generated by $\mathrm{ord}_{\widehat G}(\chi_i)$, it follows that some positive power of $x_j$ belongs to $\{x^m\colon m\in M\}$. 
Thus $x_j(v)=0$ if and only if $x_j(w)=0$, so $\support(v)=\support(w)=:J$. Take an arbitrary $G$-invariant monomial $x^n$. 
If $\support(n)\nsubseteq J$, then $x^n(v)=0=x^n(w)$. Otherwise $\support(n)\subseteq J=\{j_1<\dots <j_s\}$. Since $M_J$ generates $\mathcal{G}(\chi_{j_1},\dots,\chi_{j_s})$, there exist $u_1,\dots,u_k,t_1,\dots,t_l\in M_J$ such that $n=u_1+\dots+u_k-t_1-\dots-t_l\in \mz^s$, implying $x^nx^{t_1}\dots x^{t_l}=x^{u_1}\dots x^{u_k}$. 
By our assumption on $v,w$ we have 
$u_i(v)=u_i(w)$ for $i=1,\dots,k$ and $t_i(v)=t_i(w)\neq 0$ for $i=1,\dots,l$. It follows that 
\[x^n(v)=\frac{x^{u_1}(v)\dots x^{u_k}(v)}{x^{t_1}(v)\dots x^{t_l}(v)}=\frac{x^{u_1}(w)\dots x^{u_k}(w)}{x^{t_1}(w)\dots x^{t_l}(w)}=x^n(w).\] 
Thus we proved that $x^n(v)=x^n(w)$ holds for an arbitrary $G$-invariant monomial $x^n$, implying in turn that $h(v)=h(w)$ for any $h\in \mathcal{O}(V)^G$. 
\end{proof}

\begin{proposition}\label{prop:(i)implies(ii)} 
Let $M$ be a subset of $\mathcal{B}(\chi_1,\dots,\chi_k)$ such that $\{x^m\colon m\in M\}$ is a separating set in $\mathcal{O}(V)^G$. 
Then the abelian group $\mathcal{G}(\chi_1,\dots,\chi_k)$ is generated by $M$. 
\end{proposition}

\begin{proof} Since $e_i\in V$ can be separated from $0$ by $x^m$  for some $m\in M$, a positive power $x_i^{n_i}$ belongs to $\{x^m\colon m\in M\}$ for each $i=1,\dots,d$. 
Obviously it is sufficient to prove the statement when $G$ acts faithfully on $V$, so $G\subset GL(V)$ and hence $\widehat G=\langle \chi_1,\dots,\chi_k\rangle$. 
On the other hand $\langle \chi_1,\dots,\chi_k\rangle\cong \mz^k/\mathcal{G}(\chi_1,\dots,\chi_k)$, as $\mathcal{G}(\chi_1,\dots,\chi_k)$ was defined as the kernel of the natural surjection $\mz^k\to \widehat G$ with $e_i\mapsto \chi_i$. Denote by $H$ the abelian group $\mz^k/\mz M$. This group is finite, as $n_ie_i\in M$, hence it is isomorphic to its character group $\widehat H$. Therefore we may choose generators $\psi_1,\dots,\psi_k\in \widehat H$ such that the natural surjection $\mz^k\to \widehat H$, 
$e_i\mapsto \psi_i$ ($i=1,\dots,d$) has kernel $M\mz$. For $h\in H$ let $\rho(h)\in GL(V)$ be the linear transformation given by  $\rho(h)\xi_i=\psi_i(h^{-1})\xi_i$ where $\xi_i$ spans the summand $V_i$ in $V=\bigoplus_{i=1}^kV_i$. Then 
$\rho: H\to GL(V)$ is an injective group homomorphism. Note that for any $q\in \mn_0^k$ we have 
$x^q(\rho(h)(\xi_1+\dots+\xi_k))=\prod_{i=1}^k\psi_i(h^{-1})^{q_i}$. By the choice of $\psi_i$, for any $m\in M$ and any $h\in H$ and  we have $\prod_{i=1}^k\psi_i(h^{-1})^{m_i}=1$. 
Therefore we have 
\[x^m(\rho(h)(\xi_1+\dots +\xi_k))=1\]
for all $m\in M$ and $h\in H$. On the other hand $x^m(\xi_1+\dots + \xi_k)=1$ as well. Since $\{x^m\colon m\in M\}$ is a separating set in $\mathcal{O}(V)^G$, we 
conclude that the $H$-orbit of $\xi_1+\dots +\xi_k$ is contained in the $G$-orbit of $\xi_1+\dots+\xi_k$.  
Thus for each $h\in H$ there exists a $g\in G$ such that $\rho(h)(\xi_1+\dots+\xi_k)=g(\xi_1+\dots+\xi_k)$, implying in turn that $\rho(h)\xi_i=g\xi_i$ for each basis vector $\xi_i\in V$, and hence $\rho(h)=g$. 
So we have 
$H\cong \rho(H)\subset G\subset GL(V)$. 
Therefore $\psi_i=\chi_i\circ \rho$ for $i=1,\dots,k$, and 
it follows that the natural surjection $\mz^k\to \widehat H$, $e_i\mapsto \psi_i$ factors through the natural surjection 
$\mz^k\to \widehat G$, $e_i\mapsto \chi_i$, so $\mz M\supset \mathcal{G}(\chi_1,\dots,\chi_k)$. Now as $M$ was a subset of $\mathcal{B}(\chi_1,\dots,\chi_k)\subset  \mathcal{G}(\chi_1,\dots,\chi_k)$, the reverse inclusion $\mz M\subset  \mathcal{G}(\chi_1,\dots,\chi_k)$ also holds, forcing the equality $\mz M= \mathcal{G}(\chi_1,\dots,\chi_k)$.  
\end{proof} 
 
 \begin{lemma}\label{lemma:sep_restrict} If $\{x^m\colon m\in M\}$ is a separating set in $\mathcal{O}(V)^G$, then for all subsets $J\subset \{1,\dots,k\}$ the monomials 
 $\{x^m\colon m\in M_J\}$ constitute a separating set in $\mathcal{O}(V_J)^G$. 
 \end{lemma} 
 
 \begin{proof}  
 We claim that if $S$ is a separating set in $\mathcal{O}(V)^G$, then its restriction $\{f\vert_{V_J}\colon f\in S\}$ to $V_J$ is a separating set in $\mathcal{O}(V_J)^G$. Indeed, suppose $h(v)\neq h(w)$ for some $v,w\in V_J$ and $h\in \mathcal{O}(V_J)^G$. 
 Since the algebra $\mathcal{O}(V_J)^G$ is contained in $\mathcal{O}(V)^G$, it follows that there exists an $f\in S$ with $f(v)\neq f(w)$, so $f\vert_{V_J}$ separates $v$ and $w$. This proves the claim. 
 Now observe that if $m$ does not belong to $M_J$, then the monomial $x^m$ vanishes identically on $V_J$.  
 Consequently the restriction to $V_J$ of $\{x^m\colon m\in M\}$ is contained in  $\{x^m\colon m\in M_J\}\cup\{0\}$, and our statement follows. 
 \end{proof} 
 
 \begin{proofof}{Theorem~\ref{thm:monomial-separating}}
(i)$\Rightarrow$(ii): Suppose that $\{x^m\colon m\in M\}$ is a separating set in $\mathcal{O}(V)^G$, and take a subset $J=\{j_1<\dots<j_s\}\subset \{1,\dots,k\}$. By   
Lemma~\ref{lemma:sep_restrict} $\{x^m\colon m\in M_{J}\}$ is a separating set in $\mathcal{O}(V_J)^G$. Applying Proposition~\ref{prop:(i)implies(ii)} for $V_J$ and $M_J$ we conclude that the abelian group $\mathcal{G}(\chi_{j_1},\dots,\chi_{j_s})$ is generated by $M_J$. 

(ii)$\Rightarrow$ (iii):  Trivial. 

(iii)$\Rightarrow$(i): Suppose that (iii) holds.  Then for any subset $J=\{j_1<\dots<j_s\}\subset \{1,\dots,k\}$ with $s=|J|\le \kappa(G)$,  
the set $M_J$ generates the abelian group $\mathcal{G}(\chi_{j_1},\dots,\chi_{j_s})$, hence by Proposition~\ref{prop:(ii)implies(i)} $\{x^m\colon m\in M_J\}$ is a separating set 
in $\mathcal{O}(V_J)^G$. If $d\le\kappa(G)$, then we may take $J=\{1,\dots,k\}$ and we are done. Otherwise the union of the $\{x^m \colon m\in M_J\}$ as $J$ ranges over the 
subsets of $\{1,\dots,k\}$ of size $\kappa(G)$ is a separating set in $\mathcal{O}(V)^G$ by Lemma~\ref{lemma:helly}.  
\end{proofof} 

\begin{corollary}\label{cor:betasep} 
The number $\sepbeta(G)$ is the minimal positive integer $d$ such that for any positive integer $s\le \kappa(G)$ and any finite sequence $a_1,\dots,a_s$ of distinct elements of $G$ the abelian group $\mathcal{G}(a_1,\dots,a_s)$ is generated by 
$\{m\in \mathcal{B}(a_1,\dots,a_s)\colon |m|\le d\}$. 
\end{corollary} 

\begin{proof} For a finite abelian group $H$ denote by $\delta(H)$ (respectively $\delta_0(H)$) the minimal positive integer $d$ such that for any $s\le\kappa(H)$ and any sequence $a_1,\dots,a_s$ of not necessarily distinct (respectively distinct) elements of $H$ the  group $\mathcal{G}(a_1,\dots,a_s)$ is generated by 
the $m\in \mathcal{B}(a_1,\dots,a_s)$ with  $|m|\le d$. Obviously  $\delta_0(H)\le \mathsf D(H)\le |H|$. 
We claim that $\delta(H)=\delta_0(H)$. Indeed, the inequality $\delta_0(H)\le \delta(H)$ holds by definition of $\delta$ and $\delta_0$. To see the reverse inequality take an arbitrary sequence $a_1,\dots,a_s\in H$, $s\le\kappa(H)$. By induction on $s$ we show that  $\mathcal{G}(a_1,\dots,a_s)$ is generated by 
$\{m\in \mathcal{B}(a_1,\dots,a_s)\colon |m|\le \delta_0(H)\}$. If $a_1,\dots,a_s$ are distinct, then we are done by definition of $\delta_0$.   Otherwise suppose $a_1=a_2$. Clearly $\delta_0(H)\ge \mathrm{ord}_H(a_1)=q$, 
and $n:=(1,q-1,0,\dots,0)\in\mathcal{B}(a_1,\dots,a_s)$ satisfies $|n|\le \delta_0(H)$. 
Moreover, given any $m\in \mathcal{G}(a_1,\dots,a_s)$, replace $m$ by $\tilde m:=m-m_1n$. Then $\tilde m$ belongs to 
$\mathcal{G}(a_2,\dots,a_s)$ (viewed as the subset of $\mathcal{G}(a_1,\dots,a_s)$ consisting of the elements whose first coordinate is zero). By the induction hypothesis $\tilde  m$ belongs to the group generated by $\{m\in \mathcal{B}(a_2,\dots,a_s)\colon |m|\le \delta_0(H)\}$, implying in turn that $m$ belongs to the group generated by $\{m\in \mathcal{B}(a_1,\dots,a_s)\colon |m|\le \delta_0(H)\}$. This shows $\delta(H)=\delta_0(H)$. 

Now take a $G$-module $V$ with the notation  of the beginning of Section~\ref{sec:characterization}. 
For any subset $\{i_1<\dots<i_s\}\subset \{1,\dots,k\}$ with $s\le\kappa(\widehat G)$ the abelian group 
$\mathcal{G}(\chi_{i_1},\dots,\chi_{i_s})$ is generated by the elements 
$m\in\mathcal{B}(\chi_{i_1},\dots,\chi_{i_s})$ with  $|m|\le \delta(\widehat G)$. It follows by Theorem~\ref{thm:monomial-separating} that $\{x^m\colon m\in\mathcal{B}(\chi_1,\dots,\chi_k),\quad |m|\le\delta(\widehat G)\}$ is a separating set in $\mathcal{O}(V)^G$. Thus $\sepbeta(G,V)\le \delta(\widehat G)$. Since $V$ was an arbitrary $G$-module, we 
deduce the inequality $\sepbeta(G)\le \delta(\widehat G)$. Note finally that the isomorphism $G\cong \widehat G$ implies $\delta(G)=\delta(\widehat G)$. Combining with the first paragraph we obtain 
$\sepbeta(G)\le \delta_0(G)$. 

To show the reverse inequality $\sepbeta(G)\ge \delta_0(G)$, take 
a sequence 
 $\chi_1,\dots,\chi_k$ of characters of $G$ such that $k\le\kappa(\widehat G)$ and the abelian group $\mathcal{G}(\chi_1,\dots,\chi_k)$ is not generated by 
 $\{m\in\mathcal{B}(\chi_1,\dots,\chi_k)\colon |m|<\delta_0(\widehat G)\}$. 
 Such a sequence $\chi_1,\dots,\chi_k$ exists by definition of $\delta_0(\widehat G)$. 
 It follows by Theorem~\ref{thm:monomial-separating} that 
 $\{x^m\colon m\in\mathcal{B}(\chi_1,\dots,\chi_k), \quad |m| < \delta_0(G)\}$ is not a separating set in $\field[x_1,\dots,x_k]^G=\mathcal{O}(V)^G$, where the $G$-module $V$ is determined by $g\cdot x_i=\chi_i(g)x_i$ for $i=1,\dots,k$. Taking into account \eqref{eq:monomial-decomp} we conclude 
 $ \sepbeta(G,V)\ge \delta_0(G)$, implying in turn $\sepbeta(G)\ge \delta_0(G)$. 
 \end{proof} 


\section{Degree bounds}\label{sec:degree bounds} 

We fix the following notation for the whole Section. 
Decompose our (additively written) abelian group $G$ as a direct product of cyclic groups 
\[G=C_{n_1}\oplus \dots\oplus C_{n_r}\] 
where $n_r\mid n_{r-1}\mid\dots\mid n_1$ and $n_r>1$, so in particular $n_1$ is the {\it exponent} of $G$ and $r$ is the {\it rank} (the minimal number of generators) of $G$, hence the Helly dimension of $G$ is $\kappa(G)=r+1$. 
Set 
\[\mathsf d^*(G)=\sum_{i=1}^r(n_i-1).\] 
It is well known that 
\begin{equation}\label{eq:lowerboundfordavenport}\mathsf d^*(G)+1\le \mathsf D(G)
\end{equation}
where $\mathsf D(G)$ is the Davenport constant of $G$ (cf. Section~\ref{sec:intro}). 
Classical results in arithmetic combinatorics assert that we have equality in \eqref{eq:lowerboundfordavenport} if $G$ is a $p$-group or $G$ has rank two. On the other hand there are some infinite sequences of finite abelian groups for which the  inequality in \eqref{eq:lowerboundfordavenport} is known to be strict. 
Beyond that it is not well understood when equality holds in \eqref{eq:lowerboundfordavenport}. 
We refer to the surveys  \cite{geroldinger-gao} and \cite{Ge09a} for the above results and for  references on zero-sum sequences in finite abelian groups. 

We shall need the following technical and elementary lemma. 

\begin{lemma}\label{lemma:inequality}
Let $(n_1,\dots,n_r),(m_1,\dots,m_r)\in\mn^r$ be $r$-tuples of positive integers such that the following divisibility conditions hold for them: 
\[m_i\mid n_i \quad\mbox{ for }\quad i=1,\dots,r,\qquad n_{i+1}\mid m_i\quad \mbox{ for }\quad i=1,\dots,r-1.\] 
Then the following inequality holds: 
\begin{equation}\label{eq:inequality} 
\sum_{i=1}^r(n_i-1)\ge \sum_{i=1}^r(m_i-1)+\prod_{i=1}^r\frac{n_i}{m_i}-1. 
\end{equation}
Moreover, equality holds in \eqref{eq:inequality} if and only if there exists a $j\in\{1,2,\dots,r\}$ such that $m_1=n_1$,,\dots,$m_j=n_j$, 
$m_{j+1}=n_{j+2}$, $m_{j+2}=n_{j+3}$,  ... (where we mean that $n_{r+1}=1$).  
\end{lemma} 

\begin{proof} When $r=1$, \eqref{eq:inequality} becomes $n_1-1\ge (m_1-1)+\frac{n_1}{m_1}-1$, which is equivalent to 
the obvious $(m_1-1)(\frac{n_1}{m_1}-1)\ge 0$. Assume from now on that $r>1$. If $n_1=m_1$, then we may omit them and deal with 
the sequences $(n_2,\dots,n_r)$ and $(m_2,\dots,m_r)$, since the inequality \eqref{eq:inequality} for these shorter sequences is obviously equivalent to the corresponding inequality for the original sequences. So from now on we assume that $n_1>m_1$, that is, $m_1$ is a proper divisor of $n_1$. 

The conditions imply that $\prod_{i=1}^r\frac{n_i}{m_i}$ divides  $n_1$, and $\prod_{i=1}^r\frac{n_i}{m_i}=n_1$  if and only if $m_r=1$, $n_r=m_{r-1}$, $\dots$, $n_2=m_1$. 
Assume first that these equalities hold. Then we have 
\[\sum_{i=1}^r(n_i-1)= \prod_{i=1}^r\frac{n_i}{m_i}-1+\sum_{i=2}^r(n_i-1)=\prod_{i=1}^r\frac{n_i}{m_i}-1+\sum_{i=1}^{r-1}(m_i-1)\] 
and taking into account that $m_r=1$, we see that \eqref{eq:inequality} holds with equality in this case. 
Suppose finally that $\prod_{i=1}^r\frac{n_i}{m_i}$ is a proper divisor of $n_1$. Let $p$ be a minimal prime divisor of $n_1$. 
Then $\prod_{i=1}^r\frac{n_i}{m_i}\le \frac{n_1}p$ and $m_1\le \frac{n_1}p$. 
Consequently we have 
\[n_1-1=\frac{n_1}p\cdot p-1\ge \frac{n_1}p+\frac{n_1}p-1\ge \prod_{i=1}^r\frac{n_i}{m_i}+m_1-1>\prod_{i=1}^r\frac{n_i}{m_i}-1 +(m_1-1).\] 
Since for $i=2,\dots r$ we have $n_i-1\ge m_i-1$, we conclude \eqref{eq:inequality}. 
 \end{proof} 

Lemma 4.1 in \cite{Gr-Ma-Or09} (see also Exercise 1.6 in \cite{Gr13a}) asserts that $\mathsf d^*(G)\ge \mathsf d^*(H)+\mathsf d^*(G/H)$ for any subgroup $H$ of $G$. 
In Lemma~\ref{lemma:cyclicfactor} we provide a detailed proof of the special case when $G/H$ is cyclic, yielding also a characterization of the case when equality holds. 

\begin{lemma}\label{lemma:cyclicfactor}
Let $H$ be a proper subgroup of $G$ such that the factor group $G/H$ is cyclic. 
Then $\mathsf d^*(G)\ge \mathsf d^*(H)+[G:H]-1$, with equality  only if 
$\rank(H)=\rank(G)-1$, and $H\cong \bigoplus_{i\in \{1,\dots,r\}\setminus\{j\}} C_{n_i}$ for some $j\in \{1,\dots,r\}$.   
\end{lemma}

\begin{proof}  
Take a finite abelian $p$-group $A$. It is isomorphic to $C_{p^{\lambda_1}}\oplus \dots\oplus C_{p^{\lambda_k}}$ where 
$\lambda_1\ge \dots\ge\lambda_k>0$. We call the partition $\lambda=(\lambda_1,\dots,\lambda_k)$ the {\it type } of $A$. Any subgroup $B$ of $A$ has type $\mu=(\mu_1,\dots,\mu_k)$, 
 $\mu_1\ge \dots\ge\mu_k\ge 0$, where $\mu_i\le\lambda_i$ for $i=1,\dots,k$. 
Moreover, $A$ has a subgroup $B$ of type $\mu$ such that the factor group $A/B$ is cyclic (necessarily of order $p^d$, where  $d:=\sum(\lambda_i-\mu_i)$) 
if and only if  the Littlewood-Richardson coefficient $c^{\lambda}_{\mu,(d)}$ is non-zero (see for example  II.4.3  in \cite{macdonald}), 
and by Pieri's rules (see  for example   I.5.16 in \cite{macdonald} ) this happens if and only if the additional  inequalities 
$\mu_1\ge \lambda_2$, $\mu_2\ge \lambda_3,$, $\dots$, $\mu_{r-1}\ge\lambda_r$ hold as well. 

Note that both $G$ and $H$ are the direct products of their unique Sylow subgroups. Therefore it follows from the above paragraph that the subgroup $H$ 
is isomorphic to $H\cong C_{m_1}\oplus \dots \oplus C_{m_r}$ where $1\le m_r\mid m_{r-1}\mid \dots \mid m_1$, and $m_i\mid n_i$ for $i=1,\dots,r$. 
Moreover, since $G/H$ is cyclic, for any prime $p$ the factor of the Sylow $p$-subgroup of $G$ modulo the Sylow $p$-subgroup of $H$ is cyclic, 
so again by the above paragraph the conditions $n_{i+1}\mid m_i$ for $i=1,\dots,r-1$ hold as well.  This means that the assumptions of Lemma~\ref{lemma:inequality} hold for the $r$-tuples $(n_1,\dots,n_r)$, $(m_1,\dots,m_r)$. Thus by Lemma~\ref{lemma:inequality} the inequality \eqref{eq:inequality} holds, which is the same as the desired inequality in our statement. 
\end{proof} 

Given the finite abelian groups $G,H,K$, there exists a subgroup $G_1$ of $G$ such that $G_1\cong  H$ and $G/G_1\cong K$ if and only if there exists a subgroup $G_2$ of $G$ with $G_2\cong K$ and $G/G_2\cong H$. Therefore Lemma~\ref{lemma:cyclicfactor} has its dual form as well: 

\begin{lemma}\label{lemma:dual} 
Let $K$ be a nontrivial cyclic subgroup of $G$. 
Then $\mathsf d^*(G)\ge \mathsf d^*(G/K)+|K|-1$, with equality only if  
$\rank(G/K)=\rank(G)-1$,  and for some $j\in \{1,\dots,r\}$ we have  $G/K\cong \bigoplus_{i\in \{1,\dots,r\}\setminus\{j\}} C_{n_i}$.   
\end{lemma}

\begin{corollary}\label{cor:cyclicseries} 
Let $a_1,\dots,a_k$ be a sequence of elements generating $G$, and for $i=1,\dots ,k$ denote by $d_i$ the order of $a_i$ modulo the subgroup $\langle a_1,\dots,a_{i-1}\rangle$ (so $d_1\dots d_k=|G|$). 
Then $\mathsf d^*(G)\ge \sum_{i=1}^k (d_i-1)$, with equality only if the multiset $\{n_1,\dots,n_r\}$ coincides with the multiset $\{d_{i_1},\dots,d_{i_r}\}$ obtained by omitting all 
occurrances of $1$ in the sequence $d_1,\dots,d_k$. 
\end{corollary} 
\begin{proof} 
Apply induction for $k$. The case $k=1$ is obvious, since then $G$ is cyclic of order $d_1$, so $r=1$, $n_1=d_1$,  and $\mathsf d^*(G)=d_1-1=n_1-1$. 
Assume next that $k>1$, and set $H:=\langle a_1,\dots,a_{k-1}\rangle$. 
If $H=G$, then $d_k=1$, $\mathsf d^*(G)=\mathsf d^*(H)$, and the statement follows by the induction hypothesis applied to $H$. If $H$ is a proper subgroup of $G$, then 
$d_k>1$. By Lemma~\ref{lemma:cyclicfactor} we have $\mathsf d^*(G)\ge \mathsf d^*(H)+d_k-1$, with equality only if   
$H\cong \bigoplus_{i\in \{1,\dots,r\}\setminus\{j\}} C_{n_i}$ for some $j\in \{1,\dots,r\}$, implying also $n_j=d_k$. Now we may conclude by applying the induction hypothesis for $H$ and the sequence $a_1,\dots,a_{k-1}$. 
\end{proof}

We shall use the following terminology. Given an ordered sequence $a_1,\dots,a_k\in G$ of elements generating $G$, any $b\in G$ can be uniquely written as 
\begin{equation}\label{eq:normalform} 
b=\sum_{i=1}^k l_ia_i,\qquad 0\le l_i\le d_i-1\mbox{ for }i=1,\dots,k
\end{equation} 
where $d_i$ denotes the smallest positive integer $d$ such that $da_i$ belongs to the subgroup $\langle a_1,\dots,a_{i-1}\rangle$.  
Indeed, consider the chain 
\[\{0\}\subset \langle a_1\rangle\subset \langle a_1,a_2\rangle \subset \dots \subset \langle a_1,\dots,a_k\rangle=G \] 
of subgroups. It has cyclic factors of order $d_1,\dots,d_k$. The factor $G/\langle a_1,\dots,a_{k-1}\rangle$ is generated by the coset of $a_k$, hence 
$b+\langle a_1,\dots,a_{k-1}\rangle =l(a_k+\langle a_1,\dots,a_{k-1}\rangle)$ for a unique $0\le l\le d_k-1$. Now continue in the same way with the element 
$b-la_k$ in the group $\langle a_1,\dots,a_{k-1}\rangle$.  
We shall refer to \eqref{eq:normalform} as the {\it normal form} of $b$ with respect to $a_1,\dots,a_k$, and we call 
$\sum_{i=1}^k(d_i-1-l_i)$ the {\it deficit of $b$}. 
 
\begin{lemma}\label{lemma:d+1}
Let $a_1,\dots,a_k$ be an arbitrary sequence of elements in $G$, and denote by $d_k$ the order of $a_k$ modulo the subgroup 
$\langle a_1,\dots,a_{k-1} \rangle$. 
 Then there exists an  $m=(m_1,\dots,m_k)\in \mathcal{B}(a_1,\dots,a_k)$ such that  $m_k=d_k$ and 
 $|m|=m_1+\dots+m_k\le \mathsf d^*(G)+1$. 
 \end{lemma} 
 
 \begin{proof} The group $\langle a_1,\dots,a_{k-1}\rangle$ contains $d_ka_k$. Set  $m:=(l_1,\dots,l_{k-1},d_k)$ where  $-d_ka_k=\sum_{i=1}^{k-1}l_ia_i$ is the normal form of $-d_ka_k$ with respect to $a_1,\dots,a_{k-1}$. 
Then $m$  belongs to 
$\mathcal{B}(a_1,\dots,a_k)$, and $|m|\le \sum_{i=1}^{k-1}(d_i-1)+d_k\le \mathsf d^*(G)+1$, where the last inequality holds by Corollary~\ref{cor:cyclicseries}. 
\end{proof}

 \begin{lemma} \label{lemma:distinct} Suppose that $a_1,\dots,a_k$ are distinct non-zero elements in $G$, and denote by $d_i$ the order of $a_i$ modulo $\langle a_1,\dots,a_{i-1}\rangle$ for $i=1,\dots,k$.  
 If there is no $m\in  \mathcal{B}(a_1,\dots,a_k)$ such that 
 $|m|\le \mathsf d^*(G)$ and $m_k=d_k$, then either $k=1$ and $G=\langle a_1\rangle$, or 
 $G= \langle a_1,\dots,a_{k-1}\rangle$,  the multiset $\{d_1,\dots,d_{k-1}\}$ coincides with the multiset $\{n_1,\dots,n_r\}$ (so in particular $k-1=r$ is the rank of $G$), and the deficit of $-a_k$ with respect to $a_1,\dots,a_{k-1}$ is zero. 
 \end{lemma} 
\begin{proof} 
Suppose that $k>1$ and for the  $m$ constructed in the proof of  Lemma~\ref{lemma:d+1} we have $|m|=\mathsf d^*(G)+1$, so 
\begin{equation}\label{eq:dkak}-d_ka_k=\sum_{i=1}^{k-1}(d_i-1)a_i
\end{equation}
and $\mathsf d^*(G)=\sum_{j=1}^k(d_j-1)$.  
Assume first that for some $i\in \{1,\dots,k-1\}$ we have $d_i=1$. Then $a_i\in \langle a_1,\dots,a_{i-1}\rangle$, so we may write 
\begin{equation}\label{eq:liai}
a_i=l_1a_1+\dots+l_{i-1}a_{i-1}
\end{equation} 
in its normal form with respect to $a_1,\dots,a_{i-1}$. 
Equations \eqref{eq:dkak} and \eqref{eq:liai} imply  
\[-d_ka_k=\sum_{j=1}^{i-1}(d_j-1-l_j)+a_i+\sum_{q=i+1}^{k-1}(d_q-1)a_q.\] 
Setting $m':=(d_1-1-l_1,\dots,d_{i-1}-1-l_{i-1},1,d_{i+1}-1,\dots,d_{k-1}-1,d_k)$ we get that $m'\in \mathcal{B}(a_1,\dots.a_k)$. 
Moreover, as $a_1,\dots,a_i$ are distinct, we have that $l_1+\dots+l_{i-1}\ge 2$, hence  
$|m'|=\sum_{i=1}^{k}(d_i-1)+2-(l_1+\dots+l_{i-1})\le \sum_{j=1}^k(d_j-1)=\mathsf d^*(G)$.  So we found an $m'\in \mathcal{B}(a_1,\dots,a_k)$ with $m'_k=d_k$ and 
$|m'|\le \mathsf d^*(G)$. 

It remains to deal with the case when $d_1,\dots,d_{k-1}$ are all greater than $1$. Suppose first that $H:=\langle a_1,\dots,a_{k-1}\rangle \subsetneq G$. 
If $d_k=1$, then $\langle a_1,\dots,a_k\rangle =H\subsetneq G$, and by Lemma~\ref{lemma:d+1} there exists an $m\in \mathcal{B}(a_1,\dots,a_k)$ with $m_k=d_k$ and 
$|m|\le \mathsf d^*(H)+1\le \mathsf d^*(G)$. If 
$d_i>1$ for all  $i=1,\dots,k$,   by Corollary~\ref{cor:cyclicseries} the equality $\mathsf d^*(G)=\sum_{j=1}^k(d_j-1)$ implies that the multiset 
$\{d_1,\dots,d_k\}$ coincides with $\{n_1,\dots,n_r\}$. In particular, $k=r$ is the rank of $G$. However, since $d_1$ is the order of $a_1$, 
equation \eqref{eq:dkak} implies that  $a_1$ is contained in $\langle a_2,\dots,a_k\rangle$. Thus $G$ can be generated by $k-1=r-1$ elements. This is a contradiction, so this case does not occur. 
Finally, if $\langle a_1,\dots,a_{k-1}\rangle =G$, then $d_k=1$, and  \eqref{eq:dkak} becomes $-a_k=\sum_{i=1}^{k-1}(d_i-1)a_i$, so the deficit of $-a_k$ is zero. 
Moreover, the equality  $\mathsf d^*(G)=\sum_{j=1}^k(d_j-1)=\sum_{j=1}^{k-1}(d_j-1)$ implies by Corollary~\ref{cor:cyclicseries} that the multisets 
$\{d_1,\dots,d_{k-1}\}$ and $\{n_1,\dots,n_r\}$ coincide, finishing the proof. 
 \end{proof}
 
 \begin{lemma}\label{lemma:23} 
Let $a_1,\dots.,a_k$ be a sequence of elements of a non-cyclic group $G$, and denote by $g_i$ the order of $a_i$ modulo the subgroup $\langle a_1,\dots,a_{i-1},a_{i+1},\dots, a_k\rangle$. Assume that the following hold: 
 \begin{itemize}
 \item[(a)] There does not exist an $m=(m_1,\dots,m_k)\in\mathcal{B}(a_1,\dots,a_k)$ with  $|m|\le \mathsf d^*(G)$ such that   $m_i=g_i$ for some $i\in\{1,\dots,k\}$. 
 \item[(b)] There does not exist a pair  $m',m''\in \mathcal{B}(a_1,\dots,a_k)$ with $|m'|\le \mathsf d^*(G)$, $|m''|\le \mathsf d^*(G)$, such that $m'_i=2$ and  $m''_i=3$ for some 
 $i\in\{1,\dots,k\}$. 
 \end{itemize}
Then $2=n_{s+1}=\dots =n_{r}$ where $r=2s$ or $r=2s-1$.  
 \end{lemma}  
 \begin{proof}
No element in $\{a_1,\dots,a_k\}$ is zero, since $a_i=0$ implies  $g_i=1$, and $e_i\in\mathcal{B}(a_1,\dots,a_k)$ where $e_i$ is the $i$th standard basis vector in $\mz^k$, hence (a) implies $\mathsf d^*(G)<1$,  a contradiction. The elements $a_1,\dots,a_k$ are distinct. Indeed, assume to the contrary that  say $a_1=a_2$.  Then denoting by $d$ the order of $a_1$, we have that $g_1=1$ and $m=(1,d-1,0,\dots,0)\in\mathcal{B}(a_1,\dots,a_k)$, hence by (a) we have $\mathsf d^*(G)<d$, a contradiction (recall that $G$ is not cyclic). 

Thus $a_1,\dots,a_k$ are distinct non-zero elements of $G$, and Lemma~\ref{lemma:distinct} applies for them with an arbitrary ordering of the elements in the sequence. In particular, by condition (a) and Lemma~\ref{lemma:distinct} the rank of $G$ is $k-1$, and any $k-1$ of the elements $a_1,\dots,a_k$ generate $G$. Furthermore, after an arbitrary renumbering of the elements in the set $\{a_1,\dots,a_k\}$,  the deficit of $-a_k$ with respect to 
$a_1,\dots,a_{k-1}$ is $0$, and the multiset $\{d_1,\dots,d_{k-1}\}$ coincides with $\{n_1,\dots,n_r\}$, where $d_i$ stands for the order of $a_i$ modulo
 $\langle a_1,\dots,a_{i-1}\rangle$. 
  
We claim that for any $i\in\{1,\dots,k\}$ with $2a_i\neq 0$ there exists a $j\neq i$ such that $2a_i=2a_j$. Indeed, suppose for example that $2a_k\neq 0$. 
Recall that the deficit of $-a_k$ with respect to $a_1,\dots,a_{k-1}$ is zero, so 
\begin{equation}\label{eq:deficit}
-a_k=\sum_{i=1}^{k-1}(d_i-1)a_i.\end{equation} 
Since $-a_k\neq -2a_k$, the deficit of $-2a_k$ is different from the deficit of $-a_k$, so the deficit of $-2a_k$ is non-zero. 
Also $a_k$ is different from each of $a_1,\dots,a_{k-1}$, implying that the deficit of $-2a_k$ is not $1$. Consequently, the deficit of $-2a_k$ is at least $2$, hence 
$m':=(l_1,\dots,l_{k-1},2)$ where 
$-2a_k=l_1a_1+\dots+l_{k-1}a_{k-1}$ is the normal form of $-2a_k$ satisfies $m'_k=2$ and $|m'|=2+l_1+\dots+l_{k-1}\le \mathsf d^*(G)$. 
It follows by assumption (b) that the deficit of $-3a_k$ is at most  $2$. It can not be $0$, the deficit of $-a_k$, since $2a_k\neq 0$, and it can not be $1$, otherwise $2a_k$ coincides with one of 
$a_1,\dots,a_{k-1}$, say $2a_k=a_2$, and therefore $G=\langle a_2,a_3,\dots,a_k\rangle=\langle a_3,\dots,a_k\rangle$ is generated by $k-2=r-1$ elements, a contradiction. 
Thus the deficit of $-3a_k$ is $2$. 
There are two possible cases: $-3a_k=(d_1-2)a_1+(d_2-2)a_2+\sum_{i=3}^{k-1}(d_i-1)a_i$ or $-3a_3=(d_1-3)a_1+(d_2-1)a_2+\sum_{i=3}^{k-1}(d_i-1)a_i$ (with a suitable ordering of $a_1,\dots,a_{k-1}$). 
Comparing this with \eqref{eq:deficit} in the first case we deduce $2a_k=a_1+a_2$, hence $-a_1=-2a_k+a_2$. The latter equality shows that 
the deficit of $-a_1$ with respect to $a_k,a_2,a_3,\dots,a_{k-1}$ can not be zero, a contradiction. Thus this case does not occur. The only remaining possibility is that 
$-3a_k=(d_1-3)a_1+\sum_{i=2}^{k-1}(d_i-1)a_i$ (with a suitable ordering of $a_1,\dots,a_{k-1}$). Comparing this with \eqref{eq:deficit} we conclude 
$2a_k=2a_1$. So the claim is proved. 

It follows from the above claim that the set $\{2a_1,\dots,2a_k\}$ contains at most $\frac k2$ non-zero elements, hence the rank of the group $\langle 2a_1,\dots,2a_k\rangle 
=\{2a\colon a\in G\}$ is at most $\frac{r+1}2$.  On the other hand the rank of $\{2a\colon a\in G\}$ equals $|\{i\in\{1,\dots,r\}\colon n_i>2\}|$. Consequently we have 
$|\{j\in \{1,\dots,r\}\colon n_j=2\}|\ge \frac{r-1}2$. 
 \end{proof} 
 
  \begin{proposition}\label{prop:example} 
 Suppose that  $G=C_{n_1}\oplus \dots\oplus C_{n_s}\oplus C_2\oplus\dots\oplus C_2$ where $r=2s-1$ or $r=2s$, so  $2=n_{s+1}=\dots =n_r$, and 
 $2\mid n_s\mid n_{s-1}\mid\dots\mid n_1$. 
 Denote by $e_1,\dots,e_s,f_1,\dots,f_{r-s}$ the generators of the direct factors of $G$, thus the order of $e_i$ is $n_i$ for $i=1,\dots,s$, and the order of $f_j$ is $2$ for $j=1,\dots,r-s$. 
 Set $a_1=e_1$, $a_{2i}=e_i+f_i$ and $a_{2i+1}=f_i+e_{i+1}$ for $i=1,\dots,s-1$, and $a_{2s}=e_s$ if $r=2s-1$ whereas $a_{2s}=e_s+f_s$, $a_{2s+1}=f_s$ if $r=2s$. 
 Then  the abelian group $\mathcal{G}(a_1,\dots,a_{r+1})$ is not generated by $\{m\in \mathcal{B}(a_1,\dots,a_{r+1})\colon |m|\le \mathsf d^*(G)\}$. 
 \end{proposition} 
 \begin{proof} The element $a_1$ is contained in the group $\langle a_2,\dots,a_{r+1}\rangle$, whence there exists an element $u\in \mathcal{B}(a_1,\dots,a_{r+1})$ with $u_1=1$. 
 On the other hand we shall show that for any $m\in \mathcal{B}(a_1,\dots,a_{r+1})$ with  $|m|\le \mathsf d^*(G)$ we have that $m_1$ is even, and consequently $u$ is not contained in the subgroup of $\mathcal{G}(a_1,\dots,a_{r+1})$ generated by   
 $\{m\in \mathcal{B}(a_1,\dots,a_{r+1})\colon |m|\le \mathsf d^*(G)\}$. 
 Indeed, take $m=(m_1,\dots,m_{r+1})\in \mathcal{B}(a_1,\dots,a_{r+1})$ with 
 $m_1$ odd and $|m|$ minimal possible. 
 Consider the case when $r=2s-1$, and so $\mathsf d^*(G)=n_1+\dots+n_s-1$. The order of $a_{2i-1}$ and $a_{2i}$ is $n_i$, hence $0\le m_{2i},m_{2i-1}\le n_i-1$ hold for $i=1,\dots,s$. 
 We have 
\[0=\sum_{i=1}^{2s} m_ia_i=\sum_{i=1}^s (m_{2i-1}+m_{2i})e_i+\sum_{i=1}^{s-1}(m_{2i}+m_{2i+1})f_i.\] 
From the coefficient of $e_1$ above we deduce that $m_1+m_2=n_1$, hence $m_2$ is odd. From the coefficient of $f_1$ above we infer that $m_3$ is odd as well. 
From the coefficient of $e_2$ above we deduce that $m_3+m_4=n_2$, and consequently $m_4$ is odd. 
Continuing in the same way and looking at step-by-step  the coefficient of $f_2,e_3,f_3,e_4,\dots,f_{s-1},e_s$ we arrive at the conclusion that 
$m_{2i-1}+m_{2i}=n_i$ for all $i=1,\dots,s$, whence 
$|m|=\sum_{i=1}^s(m_{2i-1}+m_{2i})=\sum_{i=1}^s n_i>(\sum_{i=1}^s n_i)-1=\mathsf d^*(G)$.  
The case when $r=2s$ is similar. 
 \end{proof}

 \begin{proposition}\label{prop:generators} 
 Let $a_1,\dots,a_k$ be a sequence of elements of $G$. 
 \begin{itemize} 
 \item[(i)] The abelian group $\mathcal{G}(a_1,\dots,a_k)$ is generated by $\{m\in \mathcal{B}(a_1,\dots,a_k)\colon |m|\le \mathsf d^*(G)+1\}$. 
 \item[(ii)] If $r>1$ and $n_{s+1}\neq 2$ where $r=2s$ or $r=2s-1$, then 
 $\mathcal{G}(a_1,\dots,a_k)$ is generated by $\{m\in \mathcal{B}(a_1,\dots,a_k)\colon |m|\le \mathsf d^*(G)\}$. 
\end{itemize}
 \end{proposition} 
 
 \begin{proof}  (i) Take an arbitrary $u\in\mathcal{G}(a_1,\dots,a_k)$. Since $u_ka_k=-\sum_{i=1}^{k-1}u_ia_i$ belongs  $\langle a_1,\dots,a_{k-1}\rangle$, there exists an integer $l_1$ such that $u_k=l_1d_k$, where $d_k$ is the order $a_k$ modulo $\langle a_1,\dots,a_{k-1}\rangle$. By Lemma~\ref{lemma:d+1} there exists an $m^{(1)}\in \mathcal{B}(a_1,\dots,a_k)$ with $m^{(1)}_k=d_k$ and $|m^{(1)}|\le \mathsf d^*(G)+1$. Set $u':=u-l_1m^{(1)}$. Then $u'_k=0$, so $u'$ belongs to 
 $\mathcal{G}(a_1,\dots,a_{k-1})$ identified with  the subset $\{m\in\mathcal{G}(a_1,\dots,a_k)\colon m_k=0\}$ in $\mathcal{G}(a_1,\dots,a_k)$. 
 Repeat the same step for $u'$ to obtain $l_2\in \mz$ and $m^{(2)}\in \mathcal{B}(a_1,\dots,a_{k-1})$ such that $|m^{(2)}|\le \mathsf d^*(G)+1$ and 
 $u'-l_2m^{(2)}\in\mathcal{G}(a_1,\dots,a_{k-2})$. Continue in the same way, eventually we get that 
 \[u=\sum_{i=1}^kl_im^{(i)}\mbox{ where }m^{(i)}\in \mathcal{B}(a_1,\dots,a_i),\quad |m^{(i)}|\le \mathsf d^*(G)+1\mbox{ for } i\in \{1,\dots,k\}.\] 
 
 (ii) We slightly adjust the poof of (i). By our assumption on $G$, it follows from Lemma~\ref{lemma:23} that after a possible reordering of the elements $a_1,\dots,a_k$ 
and  denoting by $g_k$ the order of $a_k$ modulo $\langle a_1,\dots,a_{k-1}\rangle$ at least one of the following two possibilities holds:  
 \begin{itemize} 
 \item[(a)] there is an $m \in \mathcal{B}(a_1,\dots,a_k)$ with $|m|\le \mathsf d^*(G)$ and $m_k=g_k$; 
 \item[ (b)]  there are $m',m''\in  \mathcal{B}(a_1,\dots,a_k)$ with $|m'|,|m''|\le \mathsf d^*(G)$ and $m'_k=2$, $m''_k=3$. 
 \end{itemize} 
 Now take an arbitrary $u\in \mathcal{G}(a_1,\dots,a_k)$.  We have  $u_k=lg_k$ for some $l\in \mz$. 
 Set $u':=u-lm$ if (a) holds and $u':=u-l(m''-m')$ if (b) holds (note that in this case necessaily $g_k=1$). 
 Then $u'$ belongs to $\mathcal{G}(a_1,\dots,a_{k-1})$. Continue in the same way with the sequence  $a_1,\dots,a_{k-1}$ and $u'\in \mathcal{G}(a_1,\dots,a_{k-1})$. 
 In $k$ steps we get a presentation of $u$ as an integral linear combination of elements from  $\{m\in \mathcal{B}(a_1,\dots,a_k)\colon |m|\le \mathsf d^*(G)\}$. 
 \end{proof} 
 
 \begin{theorem}\label{thm:main} 
 For any finite abelian group $G$ we have the inequality 
 \[\sepbeta(G)\le \mathsf d^*(G)+1\] 
 with equality holding if and only if $G$ is cyclic or $2=n_{s+1}=\dots =n_r$ where $r=2s-1$ or $r=2s$. 
 \end{theorem} 
 \begin{proof}  
 Proposition~\ref{prop:generators} (i) and Corollary~\ref{cor:betasep} imply the inequality $\sepbeta(G)\le \mathsf d^*(G)+1$. 
 Furthermore,  if $G$ is not cyclic and $n_{s+1}\neq 2$ where $r=2s$ or $r=2s-1$, then by   Proposition~\ref{prop:generators} (ii) and Corollary~\ref{cor:betasep} we even get the stronger inequality 
 $\sepbeta(G)\le \mathsf d^*(G)$. 
 
 For a cyclic group $G$ any faithful $1$-dimensional $G$-module $V$ gives $\sepbeta(G,V)=|G|=\mathsf d^*(G)+1$. 
 Suppose finally that $2=n_{s+1}=\dots =n_r$ where $r=2s-1$ or $r=2s$. By Proposition~\ref{prop:example} 
 and Corollary~\ref{cor:betasep} we conclude 
 $\sepbeta(G)> \mathsf d^*(G)$. Summarizing, for these groups $G$ we have the equality 
 $\sepbeta(G)=\mathsf d^*(G)+1$. 
  \end{proof} 
  
  \begin{corollary}\label{cor:strictinequality} 
  We have the strict inequality 
  \[\sepbeta(G)<\beta(G)\] 
  for any non-cyclic finite abelian group $G$ with $n_{s+1}\neq 2$, where $r=2s-1$ or $r=2s$. 
  \end{corollary} 
  \begin{proof} Theorem~\ref{thm:main}  for a non-cyclic $G$ satisfying  $n_{s+1}\neq 2$ together with \eqref{eq:lowerboundfordavenport} and \eqref{eq:noether=davenport} yields the inequalities 
  \[\sepbeta(G)\le \mathsf d^*(G)<\mathsf d^*(G)+1\le \mathsf D(G)=\beta(G).\] 
  \end{proof} 
  
  \begin{remark}
  Since for a finite abelian group $G$ with $n_{s+1}=\dots=n_r=2$ we have $\beta_{sep}(G)=\mathsf d^*(G)+1$ by Theorem~\ref{thm:main}, therefore for such a group we have
$\beta_{sep}(G)<\beta(G)$ if and only if we have the strict inequality $\mathsf d^*(G)+1<\mathsf D(G)$. 
A complete description of the groups $G$ with $n_{s+1}=\dots=n_r=2$ and $\mathsf D(G)>\mathsf d^*(G)+1$ is not known.   
On the other hand there are infinitely many known examples of groups $G$ where $n_{s+1}=\dots=n_r=2$ both with equality $\mathsf d^*(G)+1=\mathsf D(G)$ and with strict inequality $\mathsf d^*(G)+1<\mathsf D(G)$, 
see for example Corollary 2 in \cite{geroldinger-schneider} or Corollary 4.2.3 in \cite{Ge09a}. 
  \end{remark} 
  
  \begin{center} Acknowledgement \end{center} 
  
  I thank Alfred Geroldinger for helpful comments on the manuscript. 
  

\end{document}